\documentclass[
final
]{dmtcs-episciences}

\usepackage{amsmath}

\usepackage[utf8]{inputenc}
\usepackage{subfigure}

\usepackage[round]{natbib}
\setcitestyle{authoryear,open={((},close={))}}

\author{Andrea Jim\'enez \affiliationmark{1}\thanks{Partially supported by  
CONICYT/FONDECYT/POSTDOCTORADO 3150673 and N\'ucleo Milenio Informaci\'on y 
Coordinaci\'on en Redes ICM/FIC RC130003, Chile, and 
FAPESP-Brazil Proc.~2011/19978-5.}
  \and Yoshiko Wakabayashi \affiliationmark{2}\thanks{Partially supported by
    FAPESP Project (Proc.~2013/03447-6), CNPq (Proc.~456792/2014-7, 306464/2016-0) and Project MaCLinC of
    NUMEC/USP.}}
\title[On path-cycle decompositions of triangle-free
  graphs]{On path-cycle decompositions of triangle-free
  graphs}
\affiliation{
  CIMFAV, Facultad de Ingenier\'ia, Universidad de Valpara\'iso, Chile\\
  Instituto de Matem\'atica e Estat\'istica, Universidade de S\~ao Paulo, Brazil
 }
\received{2015-6-19}

\revised{2016-9-30}

\accepted{2017-10-15}

\newcommand{\coNP}{\co{NP}}
\newcommand{\co}[1]{\mathbf{co#1}}

\newtheorem{theorem}{Theorem}
\newtheorem{corollary}[theorem]{Corollary}

\newtheorem{definition}{Definition}

\newtheorem{lemma}[theorem]{Lemma}
\newtheorem{proposition}[theorem]{Proposition}

\newtheorem{observation}{Observation}

\newtheorem{conjecture}{Conjecture}
\newtheorem{claim}{Claim}
\usepackage{graphicx,color}
\begin{document}
\publicationdetails{19}{2017}{3}{7}{659}
\maketitle
\begin{abstract}
  In this work, we study
  conditions for the existence of length-constrained path-cycle decompositions, that is, partitions 
  of the edge set of a graph into paths and cycles of a given minimum length.  
  Our main contribution is the characterization of the class of all
  triangle-free graphs with odd distance at least~$3$ that admit a
  path-cycle decomposition with elements of length at least~$4$.
  As a consequence, it follows that Gallai's conjecture on path decomposition holds 
  in a broad class of sparse graphs. 

\end{abstract}

\section{Introduction}~\label{theo:main}

All graphs considered here are simple, that is, without loops and
multiple edges. As usual, for a graph~$G$, we denote by $V(G)$ and $E(G)$ its
vertex set and edge set, respectively.  A collection of
subgraphs~$\mathcal{D}$ of $G$ is called a \emph{decomposition of
  $G$} if each edge of $G$ is
contained in exactly one subgraph of~$\mathcal{D}$.  If a
decomposition $\mathcal{D}$ of $G$ consists only of paths
(resp. cycles), then we say that it is a \emph{path decomposition}
(resp. \emph{cycle decomposition}) of $G$, and if it consists of paths
and cycles, then we say that it is a \emph{path-cycle decomposition}.

About fifty years ago, according to Lov\'asz~\cite{Lov68}, Gallai
conjectured the following bound on the cardinality of a path
decomposition of a connected graph:

\begin{conjecture}[Gallai's Conjecture]
  {Every connected graph $G$ on $n$ vertices has a path
    decomposition of cardinality at
    most~$\left\lceil\frac{n}{2}\right\rceil$.}
\end{conjecture}
Despite many efforts and attempts~\cite{JGT:JGT3190040207,Fan2005117,JGT:JGT21735,Lov68,Pyber1996152} 
to prove Gallai's conjecture, it
remains unsolved.  
The earliest and major progress towards this conjecture
was made by Lov\'asz. In~\cite{Lov68}, he proved that every graph on $n$ vertices has a
path-cycle decomposition of cardinality at
most~$\lfloor{n}/{2}\rfloor$.
Since any path-cycle decomposition of a graph has at least $p/2$ paths, where $p$ is the number of odd vertices in the graph, Lov\'asz's result implies that graphs with at most one even vertex
satisfy Gallai's conjecture; an \emph{even vertex} (resp. \emph{odd vertex}) is a vertex with even degree
(resp. odd degree). 
Dean and Kouider~\cite{DEAN200043} proved that every graph~$G$
has a path decomposition into at most $\frac{p}{2} + \left \lfloor \frac{2}{3} q\right \rfloor$ paths,
where $p$ (resp. $q$) is the number of odd (resp. even non-isolated) vertices in $G$.
Later, Harding and McGuinness~\cite{JGT:JGT21735} proved that this bound can be greatly improved 
for graphs of large girth. They showed the bound $\frac{p}{2} + \left \lfloor \frac{g+1}{2g} q\right \rfloor$, where
$g\geq 4$ denotes the girth of the graph~$G$.
As these works suggest, it seems to be particularly difficult to guarantee
the validity of Gallai's conjecture on graphs with many even vertices.
Indeed, the broadest subclass of Eulerian graphs, excluding the complete
graphs on an odd number of vertices, for which the
conjecture is known to be true is the one composed of graphs of
maximum degree~4~(see~\cite{FKfav}).


%
%

Our work contributes in several directions.
Nevertheless, it is interesting in its own right.
  Let $\mathcal{F}$ denote the set of paths and cycles of length {(number of edges)} at least~4.
  Our main result is the characterization of the class of
  triangle-free graphs where odd vertices are at distance at least three that do not admit  
 a decomposition into copies of graphs from~$\mathcal{F}$.
  We call those graphs \emph{hanging-square graphs}; which among others,
  satisfy Gallai's conjecture and can be 
  recognized in polynomial time.
%
 As a consequence, we can guarantee {that Gallai's
 conjecture holds for graphs} with a large number of even vertices and linear number of edges.  
 In particular, {we verify}
  the conjecture for a subclass of planar graphs.  
  In the following subsection, we formalize our results.

\noindent

\subsubsection*{Contributions}

The \emph{odd distance} of a graph $G$, denoted by $d_o(G)$, is the minimum distance 
between any pair of odd vertices of~$G$.
A path-cycle decomposition $\mathcal{D}$ of $G$ is called a \emph{4-pc decomposition} if 
every element of $\mathcal{D}$ has length at least~4; in other words, a $4$-pc decomposition
is a decomposition into copies of graphs from~$\mathcal{F}$.
We recall that a graph is \emph{triangle-free} if it does not have cycles of length~$3$.

{We focus our studies on the following set of graphs:}
%
$$\mathcal{G}:=\{G: G \hbox{ is a connected triangle-free graph with  } d_o(G)\geq 3 \}.$$

{We note} that $\mathcal{G}$ corresponds to the family of graphs that can be obtained
from any triangle-free Eulerian graph by removing a matching $M$ (possibly empty) that 
satisfies the following property: for every pair $e, e'
\in M$, the minimum distance between an end vertex of $e$ and an end vertex of $e'$
is {at least} $3$.

The main contributions of this work {are} the next theorem and some of its consequences. 
(For the definitions of \emph{hanging-square graph} and \emph{skeleton}, see Section~\ref{sec:h-sgraphs}.)
%
%
{
\begin{theorem}~\label{theo:main}
A graph in $\mathcal{G}$ has a 4-pc decomposition if and only if it is not a hanging-square graph.
\end{theorem}
}
%
As a corollary, we have the following statements. 
%
%
{
\begin{corollary}~\label{coroll:main} Every graph in $\mathcal{G}$ on $n$ vertices with at most $4\,\lceil{n}/{2}\rceil$ 
edges has a path decomposition 
  into at most $\lceil{n}/{2}\rceil$ paths.
\end{corollary}
}
From the proof of Corollary~\ref{coroll:main}, one can also obtain that each path 
of the decomposition has length at least 3, which is best possible.
%
%
{Gallai's conjecture is open in the class of planar graphs.
This is quite surprising if we consider that
Haj\'os' conjecture, which states that every Eulerian graph  on $n$ vertices has a cycle
decomposition of cardinality at most~$\lfloor{n}/{2}\rfloor$,
has been positively settled for planar graphs~\cite{Seyffarth1992291}.
Since a planar triangle-free graph on $n$ vertices has at most $2n-4$ edges, 
the next result follows immediately from Corollary~\ref{coroll:main}.}

{
\begin{corollary}~\label{coroll:mainsec} Every planar graph in $\mathcal{G}$ satisfies Gallai's conjecture.  
\end{corollary}
}

We believe
  that this work contributes with a substantial step towards showing {that 
  Gallai's conjecture holds} for the class of planar graphs.

%
As we will see, the class of hanging square graphs can be defined
recursively, and for that, we define first the subclass of such graphs
that are acyclic, and name them skeletons. For them, the following holds.

{
\begin{corollary}~\label{coroll:mainskeleton} Every tree in $\mathcal{G}$ either has a decomposition into paths of length at least~4 or is a skeleton. 
\end{corollary}
}

{We observe that the skeletons can be recognized in polynomial time. 
Thus, in view of the above corollary, the \emph{certificate} that a 
tree is a skeleton (its \emph{building sequence}, as we will define) can be seen as a short certificate that it cannot be decomposed into paths of length at least~$4$.}
At the end of Section~\ref{sec:concludremarks}, we discuss a polynomial time strategy
to recognize hanging-square graphs.


\subsubsection*{Organization of the paper}

In Section~\ref{sec:h-sgraphs}, we define the class of {skeletons and} hanging-square graphs.  
Section~\ref{sec:proofth1} is devoted to the proof of Theorem~\ref{theo:main}.  
In Subsection~\ref{sec:cor2} we show Corollary~\ref{coroll:main}. 
Section~\ref{sec:proph-sgraphs} contains several properties regarding hanging-square graphs
which are fundamental for the proofs of Theorem~\ref{theo:main}.
Finally, Section~\ref{sec:concludremarks} contains some concluding remarks.


\section{Hanging-square graphs}\label{sec:h-sgraphs}

{Throughout this paper, we denote a path of length~$k$ (number of its edges) by a sequence of its
  vertices, as for example, $P=x_0 x_1\cdots x_k$ and say that it is a \emph{$k$-path}.}

In the following, we introduce a class of trees, called \emph{skeletons}. For that,
we consider $k$-paths $x_0\cdots x_k$, with $k \in \{3,4,6\}$,
and 
for each $k \in \{4,6\}$, we
say that (the middle vertex) $x_{\frac{k}{2}}$ is the \emph{joint} of
$P$.

\begin{definition}[Skeletons]\label{def:ske}
  A \emph{skeleton} is a  tree $T$ that admits a black-red
  coloring~$\lambda$ of $V(T)$, a sequence $T_0,T_1, \ldots, T_t$ of
  trees, and a sequence $P_1, \ldots,P_t$ of paths, each of which is a $4$-path or a $6$-path, 
  such that: 
  \begin{itemize}
  \item $T_0$ is a $3$-path, $T_t = T$, and
  for $i\in[t]$ the tree $T_i$ is obtained from $T_{i-1}$ by
  adding $P_i$ so that the joint of $P_i$ is identified
  with a vertex of $T_{i-1}$;
   \item for $T_0=x_0x_1x_2x_3$, we have $\lambda(x_0)= \lambda(x_3)=$ black 
  and $\lambda(x_1)= \lambda(x_2)=$ red;
  \item for each $P_i=x_0\cdots x_k$, $i\in[t]$, it holds that
  $\lambda(x_0)= \lambda(x_k)=$ black and $\lambda(x_1)= \lambda(x_{k-1})=$ red.
  If $k=4$, then $\lambda({x_{2}})=$ red. 
  If $k=6$, then $\lambda(x_{3})=$ black 
  and $\lambda(x_2)=\lambda(x_4)=$ red.
  \end{itemize}
\end{definition}
\noindent
In Figure~\ref{fig:constructionske} we show an example
of a skeleton $T$, obtained from a sequence $T_0, T_1, \ldots, T_7=~T$.

The following observation  helps understanding skeletons.

\begin{observation}\label{obs:black-red}
  Let $T$ be a skeleton and $\lambda$, $\lambda'$ be colorings of $V(T)$ as in Definition~\ref{def:ske}.  
  Then, $\lambda=\lambda'$, the set of odd vertices of $T$ is
  $\{v: \lambda(v) = \text{black}\} $ and the set of even vertices of $T$ is
  $\{v: \lambda(v) = \text{red}\}$.
 \end{observation}

The sequence $P_0, P_1, \ldots, P_t$, where $P_0=T_0$, is called a
\emph{building sequence} of $T$ and each $P_i$ a \emph{building path}. 
We might denote the skeleton $T$ by its building sequence $P_0, P_1, \ldots, P_t$.
Every time we consider a skeleton $T$,
we {implicitly} assume that it comes with a black-red coloring $\lambda$, as defined above.
The following result is not needed in what follows, but it 
is a noteworthy property of skeletons.

\begin{observation}
If $T$ is a skeleton, then all building sequences of~$T$ have the same number of building paths.
\end{observation}

We refer to each 3-path of $T$ colored black-red-red-black as a \emph{brrb-path}. Those paths
play a fundamental role throughout this work. In particular,  
the following holds: for every brrb-path~$P$ in $T$, there is a
building sequence of~$T$ that starts at~$P$ (see~Proposition~\ref{prop:fromanypath}). 

 \begin{figure}[h]
 \centering
 \ifpdf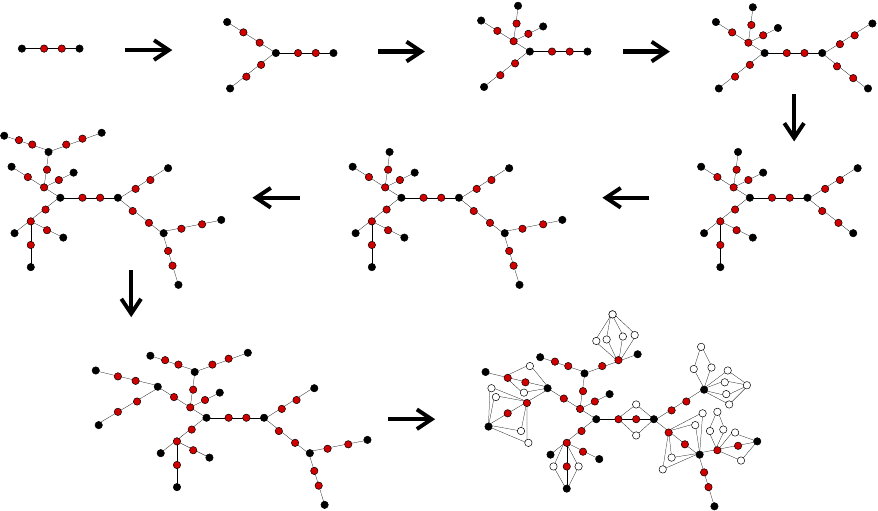\else\input{construction.ps_t}\fi
 \caption{Construction of a skeleton $T$ obtained from
  the sequence $T_0, T_1, \ldots, T_7=T$. It has a building sequence $P_0,P_1,\ldots, P_7$, where $P_0 = T_0$, $P_2$ and $P_4$ are $4$-paths, and $P_1$,$P_3$,$P_5$ and $P_7$ are $6$-paths. The graph $H$ is an example of a hanging square graph with skeleton $T$.}
 \label{fig:constructionske}
 \end{figure}

For simplicity, we may refer to a cycle of length~$4$ as a \emph{square}.  

\begin{definition}[Bunch of squares]\label{def:blocks} 
  Let $k \geq 1$. A \emph{bunch $B$ of $k$ squares} (or simply, a \emph{bunch}) is a graph obtained from the union of $k$ pairwise
  edge-disjoint squares, say $Q_1,\ldots,Q_k$, such that each of these squares contains
  two non-adjacent vertices $a$, $b$,
  and $V(Q_i)\cap V(Q_j)=\{a,b\}$ for $1\leq i<j\leq k$. 
\end{definition}

In other words, a {bunch $B$ of $k$ squares} is 
a complete bipartite graph $K_{2,2k}$, where one of the partition classes consists
of $a$ and $b$.

{The common non-adjacent vertices $a$, $b$} of a bunch are called \emph{joints}. In the case that 
the bunch is a square (namely, $k$=1 in Definition~\ref{def:blocks}) there are two pairs 
of joints. For consistency, we consider only one of them as joints. Thus, 
each bunch has exactly one pair of joints. 

Let $\mathcal{B}$ be a set of bunches.
We say that a bunch $B \in \mathcal{B}$ with joints $a, b$ is \emph{maximal} if 
there is no $B' \in \mathcal{B}$ distinct of $B$ with joints $a, b$.

\begin{definition}[Hanging-square graphs]\label{def:h-sss} 
A graph $H$ is a hanging-square graph if it is the union of a skeleton $T_H$
and a set $\mathcal{B}$ of maximal bunches so that the following holds.
\begin{itemize}
\item[(i)] For each $B \in \mathcal{B}$, the vertices in $V(B) \cap V(T_H)$ 
are joints of $B$;
accordingly, if $|V(B) \cap V(T_H)|=1$, then
we call $B$ a \emph{1-bunch}, and if $|V(B) \cap V(T_H)|=2$, then
we call $B$ a \emph{2-bunch}. Moreover, if $B' \in \mathcal{B}$ and $V(B) \cap V(B') \neq
\emptyset$, then $V(B) \cap V(B') \subseteq V(T_H)$.

\item[(ii)] If $B$ is a 2-bunch, then there exists a brrb-path $P_B=x_0x_1x_2x_3$ in $T_H$  
such that $x_0, x_2$ are the joints of $B$. Moreover, 
if $B' \in \mathcal{B}$ and $B \neq B'$, then $P_B \neq P_{B'}$. We say that the path $P_B$ is \emph{occupied} 
at $x_0, x_2$ by $B$. 

\item[(iii)] If $P_B = x_0x_1x_2x_3$ is a brrb-path occupied at $x_0,x_2$ by a 2-bunch $B$, then $d_H(x_1) = 2$. 
Moreover, for $H' = H-E(B)$, we have 
either $d_{H'}(x_0) = 1$
or $d_{H'}(x_2) = 2$.
  \end{itemize}
 
\end{definition}

In particular, in Definition~\ref{def:h-sss}, item~(ii) says that
each brrb-path is occupied either at none, or at 2 vertices and, 
 item~(iii)  
 says that $x_1$  is neither the joint of a 1-bunch, nor 
 the joint of a building path of $T_H$, and that 
 joints of 1-bunches and joints of building paths of $T_H$ may
 correspond either to $x_0$ and $x_3$, 
 or to $x_2$ and $x_3$ (namely, such joints cannot simultaneously correspond to $x_0$ and $x_2$).

 \medskip
 
We now prove that hanging-square graphs satisfy Gallai's conjecture.
For that, we use a result of Fan~\cite{Fan2005117}. 
A \emph{block} of a graph is a maximal connected subgraph without a cut vertex 
(i.e. a vertex whose deletion increases its number of connected components). 
Given a graph $G$, the \emph{even graph} of $G$ is the subgraph of $G$ induced by the vertices
of even degree in $G$.

\begin{theorem}[Corollary of the Main Theorem of~\cite{Fan2005117}]\label{theo:fan}
Let $G$ be a graph. If each block of the even graph of $G$ is a triangle-free graph of maximum degree at most 3,
then $G$ satisfies Gallai's conjecture.
\end{theorem}

\begin{proposition}\label{obs:gallaiforhs}
If $H$ is a hanging-square graph, then $H$ satisfies Gallai's conjecture.
\end{proposition}
\begin{proof} 
Let $H$ be a hanging-square graph and $T_H$ be its skeleton. 
If $C$ and $C'$ are squares such that $V(V)\cap V(C')\neq \emptyset$, $E(V)\cap E(C')= \emptyset$,
and $C\cup C'$ is triangle-free, then one can easily find a decomposition of $C\cup C'$ into two paths.
Hence, we can assume that $H-E(T_H)$ is a union of isolated squares and vertices.

Let $H'$ be the even graph of $H$.
On the one hand, if $v \in V(H') \cap V(T_H)$ and the degree of $v$ is at least 2, then   
$v$ is a cut vertex of $H'$. 
On the other hand, if $v \in V(H') - V(T_H)$, then $v$ is a vertex of a square.
Hence, each block of $H'$ is a vertex, an edge or a square.
Due to Theorem~\ref{theo:fan}, graph $H$ satisfies
Gallai's conjecture. 
\end{proof}

%
%
%

%
%
%

\section{Proof of the main results}\label{sec:proofth1}

In this section we discuss the proof of Theorem~\ref{theo:main}
and of Corollary~\ref{coroll:main}.
Recall that a 4-pc decomposition is a decomposition into paths and cycles whose lengths
are at least 4. 
The main result of this work
is a characterization of the class of connected triangle-free graphs 
with odd distance at least $3$ having no 4-pc decompositions. 
We prove that this class is exactly the class of all hanging-square graphs.
The proof of Theorem~\ref{theo:main} relies on the following two facts:\vspace*{-0.15cm}
\begin{enumerate}
 \item[Fact 1.] Each graph in $\mathcal{G}$ is a hanging-square graph, or 
 has a 4-pc decomposition (Theorem~\ref{prop:tree}).\vspace*{-0.15cm}
 \item[Fact 2.] Hanging-square graphs do not admit a 4-pc decomposition (Proposition~\ref{prop:hsgrhno4}).
\end{enumerate}

The proof of Fact 1 needs the use of some technical results regarding the structure of the 
hanging-square graphs {which are postponed to Section~\ref{sec:proph-sgraphs}.}

In Section~\ref{subsec:skeleton}, we 
introduce some notions and results that are helpful in 
the proofs. In particular, at the end of the subsection, we prove the first
step to complete Theorem~\ref{prop:tree}. Namely, we show that trees in $\mathcal{G}$ are skeletons, or 
 have a 4-pc decomposition.

\subsection{Properties of skeletons}\label{subsec:skeleton}

To simplify notation, if $E'$ is a a subset of the edge set of a graph
$G$, we denote by $G-E'$ the graph obtained from $G$ by first removing
the edges from $E'$ and then deleting all isolated
vertices. Consistently, whenever we refer to the deletion of an edge
set of a graph, we also consider that the resulting graph has no
isolated vertices.
 
 \begin{observation}\label{obs:redblneigh}
      Let $T$ be a skeleton. Then, each red vertex in  $T$ has a neighbor that is black.
 \end{observation}
 \begin{proof}
   Let $P_0, P_1, \ldots, P_t$ be a building sequence of $T$.  We
   proceed by induction on $t$. For $t=0$ the statement trivially
   holds.  We assume that $t>0$. Suppose $P_t$ is a 4-path. The
   joint of $P_t$ is red and is a vertex of the skeleton $P_0, P_1, \ldots, P_{t-1}$; thus, by
   induction hypothesis the joint has a neighbor that is black.
   {If $P_t$ is a 6-path, the proof follows} by induction
     hypothesis and as a direct consequence of the coloring of~$P_t$.
\end{proof}

Let $P$ be a building path of a skeleton; that is, a 4-path or a 6-path. We write $P=P' \sqcup \tilde{P}$
if $\{P'$, $\tilde{P}\}$ is a decomposition of~$P$
into two paths of equal length. Note that both, $P'$ and $\tilde{P}$, have as an end vertex the joint of~$P$.

\begin{observation}\label{obs:1}
  {Let $T$ be a skeleton.  
  Then each edge of $T$ belongs to a brrb-path~$P$.}
\end{observation}
\begin{proof}Let $\lambda$ be the coloring of $T$ and $P_0, P_1, \ldots, P_t$ be a building sequence of $T$ such that
  $P_i=P'_i \sqcup \tilde{P}_i$.  Let $e \in E(T)$. If $e \in E(P_0)$, then $P=P_0$. 
   If $e \in P_i$ for some $i$, and $P_i$ is a 6-path, then $P \in
  \{P'_i, \tilde{P}_i\}$.  If $P_i$ is a 4-path, {according to
    the proof of Observation~\ref{obs:redblneigh}, the joint $v$ of
    $P_i$ is adjacent to a vertex $u$ (that belongs to the skeleton $P_0, P_1,  \ldots, P_{i-1}$)  such that $\lambda({u})$
    is black.} In this case, $P\in \{P'_i\cup vu, \tilde{P}_i\cup vu\}$.
\end{proof}

The following is a key property that will be used many 
times throughout this paper. 

\begin{proposition}\label{prop:fromanypath}
  Let $T$ be a skeleton.  
  For every brrb-path $P$
  in $T$, there is a
  building sequence of~$T$ that starts at~$P$. 
\end{proposition}
\begin{proof}
  We proceed by induction on $|E(T)|$. In case that $T$ is a 3-path the
  result is trivial.  Assume $|E(T)|>3$ and let $P_0, P_1, \ldots,
  P_t$ be a building sequence of $T$.  By induction hypothesis, the
  statement holds for the skeleton $T'$ with building sequence
  $P_0, P_1, \ldots, P_{t-1}$.  Therefore, it
  suffices to prove the statement in the case that $P$ is a brrb-path of
  $T$ such that $E(P)\cap
  E(P_{t})\neq \emptyset$; otherwise, the result follows
  directly from the induction hypothesis.

  We suppose first that $P_{t}=P'_{t}\sqcup\tilde{P}_{t}$ is a 6-path.
  Then, $P$ is either $P'_{{t}}$ or $\tilde{P}_{{t}}$.  Without loss of
  generality, let $P=P'_{{t}}$. Let $v$ denote the joint of
  $P_{{t}}$.  Given that $v$ is a black vertex in~$T'$, by
  Observation~\ref{obs:1}, $T'$ contains a $3$-path $P'$ with
  end vertex $v$. By induction hypothesis, $T'$ has a building
  sequence that starts at $P'$, say $P', P'_1, \ldots,
  P'_{t-1}$. Hence, $P'_{t}, (P'\cup \tilde{P}_{t}), P'_1,
  \ldots, P'_{t-1}$ is a building sequence of $T$.

  Secondly, we suppose that ${P}_t=P'_t\sqcup\tilde{P}_t$ is a
  4-path and let $v$ be its joint.  Without loss of generality,
  we can assume that $P = P'_t \cup vu$, for some neighbor $u$ of
  $v$ in $T'$. As before, by Observation~\ref{obs:1} and induction
  hypothesis, there exists a brrb-path $P''$ containing the edge
  $vu$, and furthermore
  $P''$ is the starting path of a building sequence of $T'$, say 
\hbox{$P'', P'_1, \ldots, P'_{t-1}$}.  Hence, $(P'_t\cup vu),
  ([P''-vu]\cup \tilde{P}_t), P'_1, \ldots, P'_{t-1}$ is a building
  sequence of~$T$.
\end{proof}

Finally, we prove that trees in $\mathcal{G}$ are skeletons or have a $4$-pc~decomposition.
We first make the following observation.

\begin{observation}~\label{obs:forprop1}
Let $T$ be a skeleton and $u$ be an odd degree vertex of $T$.
By Observation~\ref{obs:1} and Proposition~\ref{prop:fromanypath}, 
there exists a brrb-path $P$ with end vertex $u$ 
such that $P$ is the starting path of a building sequence of $T$.
Therefore, the graph $T \cup uv$, where $v$ is a vertex not in $V(T)$,
has a 4-pc~decomposition, which is given by $P\cup uv$ and the set of paths
of the aforementioned building sequence of $T$.
\end{observation}

\begin{lemma}\label{theo:hangtrees}
 Let $T$ be a tree in $\mathcal{G}$. Then $T$  has a 4-pc~decomposition
 or is a skeleton.
\end{lemma}
\begin{proof}
If $T$ is a minimum counterexample to Lemma~\ref{theo:hangtrees}, then the following two properties hold.
\begin{description}
\item{\emph{Claim~1}:} 
 \textsl{$T$ contains no leaf $v$ such that the minimum distance from
    $v$ to any other odd vertex of $T$ is at least~4.}  
    
    If
  such a leaf $v$ exists, then the tree $T-vu$, where $u$ is the only
  neighbor of $v$, has odd distance at least~3. By the minimality of
  $T$, we conclude that $T-vu$ has a $4$-pc~decomposition or is
  a skeleton.  In the former case, since $u$ is odd in $T-vu$, we
  can trivially extend the $4$-pc~decomposition of $T-vu$ to one of
  $T$, a contradiction. If $T-vu$ is a skeleton, by
  Observation~\ref{obs:forprop1},  $T$ has a $4$-pc~decomposition, again a contradiction.


\item{\emph{Claim~2}:} 
\textsl{$T$ contains no path $P$ of length at least~4 such that
    $T-E(P)$ is connected and both leaves of $P$ have degree~1 in
    $T$.}  
    
    If such a path $P$ exists, then $T-E(P)$ is a tree
  of odd distance at least~3. If $T-E(P)$ has a 4-$pc$ decomposition,
  clearly $T$ also has such a decomposition, a contradiction. Thus, by
  the minimality of~$T$, we conclude that $T-E(P)$ is a skeleton,
  in which case, by Lemma~\ref{lemma:treepath}, we have that $T$ has a $4$-pc~decomposition, a contradiction.
  
\item{\emph{Claim~3}:} 
\textsl{If $P$ is a 3-path with odd end vertices, then $T-E(P)$ has a {{$4$-pc~decomposition}}
and is connected.}  

Let $P=v_0v_1v_2v_3$. The graph $T-E(P)$ has at most $4$~components, say
{$H_0$, $H_1$, $H_2$ and $H_3$,} containing $v_0$, $v_1$, $v_2$, and $v_3$,
respectively. Set $H_i=\emptyset$ if $H_i$ does not exist.
Notice that none of the $H_j$ can be a skeleton because the
distance between odd vertices needs to be at least 3. To see
this, notice that in $H_i$ the vertex $v_i$ has even degree. If $H_i$ is a
skeleton, then $v_i$ has a neighbour $u_i$ in $H_i$ of odd degree (by
Observation~\ref{obs:redblneigh}), but then $u_i$ has distance at most 2 from $v_0$ or $v_3$.
Hence, every component has a $4$-pc~decomposition.  
If there are at least two such components, then {$P \cup H_i$ is a skeleton} 
for each $i\in \{0,1,2,3\}$; this follows because 
$P \cup H_i$ has {fewer} edges than $T$ and if it had a $4$-pc~decomposition, 
then $T$ would have a $4$-pc~decomposition as well, a contradiction.  
But this means that $P \cup H_0 \cup H_1 \cup H_2 \cup H_3$ is a skeleton. 
Thus, there must be only one component.  
\end{description}

Let $v_0$ be a leaf of $T$ and $v_3$ be an odd vertex of $T$ at
distance~$3$ from $v_0$ in $T$ (by \emph{Claim~1}, such vertices always exist).  Let $P'=v_0v_1v_2v_3$ denote the
$3$-path of $T$ with end vertices $v_0,v_3$ and let $H_i$ with $i\in\{1,2,3\}$ the only component
of $T-E(P')$ which is not empty, as in \emph{Claim~3}. 
Due to Claim~3, $H_i$ has a $4$-pc~decomposition.

Let $\mathcal{D}$ be a $4$-pc~decomposition of $H_{i}$ such that the paths containing $v_i$ are as long as possible.
Because of maximality, every such path ends in odd vertices of $T$.
If $i=1$ (analogously for $i=2$), 
then every path in $\mathcal{D}$
containing $v_1$ is of length 4 and has~$v_1$ as its middle vertex, 
this is because $T$ does not have a 4-pc decomposition. Let $x_0x_1x_2x_3x_4$ denote such a path. 
Due to Claim~3, $x_0$ and $x_4$ are leaves of $T$ 
and the degree of its inner vertices distinct of $v_1$, namely of $x_1$ and $x_3$, is two. 
But then, we obtain that the path $x_0x_1v_1v_2v_3$ violates \emph{Claim~2}. 
For the case $i=3$, note that again due to that $T$ does not have a 4-pc~decomposition,
every path in $\mathcal{D}$
that contains $v_3$ is of length 6 and has~$v_3$ as its middle vertex. Let $x_0x_1x_2vy_2y_1y_0$ denote such a path.
Due to Claim~3, the end vertices $x_0$ and $y_0$ are leaves of $T$ and the degree of the inner vertices distinct from $v_3$, namely
of $x_1, x_2, y_1$ and $y_2$, is two. Then, the path $P'\cup x_0x_1x_2v$ violates \emph{Claim~2}.
\end{proof}

\subsection{Minimum counterexample argument}\label{subsec:mincount}

In this subsection we study the properties of a minimum counterexample to Fact~1.  
Due to Lemma~\ref{theo:hangtrees}, we already know that 
a minimum counterexample has cycles.  
As we already mentioned, the proof of this theorem is based on the results {that will be presented in} 
Section~\ref{sec:proph-sgraphs}.  Recall that $\mathcal{G}$ denotes the set of all connected, triangle-free 
graphs with odd distance at least 3.  From now on, we may denote the number of edges of a graph $G$ by $\ell(G)$.

\begin{theorem} \label{prop:tree}
Let $G \in \mathcal{G}$. If $G$ is not a hanging-square graph, then $G$ has a $4$-pc decomposition.
\end{theorem}
\begin{proof}
For the purpose of contradiction, let $G\in \mathcal{G}$ be an edge-minimum graph that is not a hanging-square 
  graph and does not have a $4$-pc decomposition.
Due to Lemma~\ref{theo:hangtrees}, $G$ has a cycle.
Since $G$ does not
  have a $4$-pc~decomposition, for every cycle $C$ of $G$, the graph
  $G-E(C)$ consists of at least one component, and by the minimality of $G$ at least one
  component of $G-E(C)$ is a hanging-square graph.
  Recall that if $E' \subset E(G)$, we denote by $G-E'$ the graph obtained from $G$ by first removing
  the edges from $E'$ and then deleting all isolated vertices.
  We analyze two cases. \\

\noindent
\textit{Case 1.} There exists a cycle $C$ in $G$ such that $G-E(C)$ has
exactly one component.

Let $H=G-E(C)$. By the observation above, $H$ is a hanging-square graph.  We claim that $C$ has length~$4$. 
Indeed, if the length of $C$ is at least~5, by 
{Lemma~\ref{lemma:length5}} we have that $G=C\cup H$ has a $4$-pc~decomposition, a contradiction.

Suppose that there exists a square $Q$ in $H$ (such a square is part of a bunch or it is a bunch itself).  We have that $G-E(Q)$
either has exactly one component or, is formed by two
components, namely, $C$ and $H-E(Q)$.  In the case that
$G-E(Q)$ is formed by two components, by Lemma~\ref{lemma:length4} we
have that $G$ has a $4$-pc~decomposition, again a contradiction.
In the case that $G-E(Q)$ has one component,
say $H'$, given that $G$ does not have a $4$-pc decomposition, we have
that $H'$ is a hanging-square graph, but then by
Lemma~\ref{lemma:2cycles} there exists a $4$-pc~decomposition of~$G$, a contradiction. 

Thus, we conclude that $H$ has no squares, and therefore $H=T_H$ (that is, $H$ is a skeleton).  
If $H$ is a 3-path, then $G = H \cup C$ is a hanging-square graph, a contradiction to
our assumption. Let $P$ be the last path in a building sequence of $T_H$. Since $P$ has
length 4 or 6 and $G$ has no $4$-pc decomposition, we have that also $G-E(P)$ has
no 4-pc decomposition. By minimality of $G$, the graph $G-E(P )$ either has two
connected components, namely $H-E(P)$ and $C$, or $G-E(P)$ is a hanging-square
graph.
Hence, by Lemma~\ref{lemma:main} we have that $G=H\cup C$ has a $4$-pc~decomposition, again a contradiction.  Summarizing, we conclude
that \emph{Case 1} leads to a contradiction.\\

\noindent
\textit{Case 2.} The deletion of the edge set of any cycle 
gives at least two components. 

Let $C$ be a cycle of $G$ of minimum length {and let $H=G-E(C)$.}

We first suppose that one of the components of $H$, say
$K$, is not a hanging-square graph. By the minimality of $G$,
the graph~$K$ has a $4$-pc~decomposition.  Let $H'= G- E(K)$.  Since $G$
is a minimum counterexample, $H'$ is a hanging-square graph.  
Moreover, since $H'-E(C)$ is connected, $H$ has exactly two components, namely $K$ and $H'-E(C)$,
and since $H'$ is a hanging-square graph, $C$ is a square of $H'$.
In addition,  given that $C$ disconnects $G$, we have $V(K) \cap V(H') \subset V_c $, where $ V_c= V(C)\setminus V(T_{H'})$.
Let us consider a $4$-pc~decomposition of $K$ such that all its paths
have length at most~$7$.  Let $D$ be an element
of such a $4$-pc~decomposition that contains a vertex of $C$.  {If
  $D$ is a cycle, then by Lemma~\ref{lemma:length4} we have that
  the graph $H' \cup D$ has a $4$-pc~decomposition. If
  $D$ is a path, then Lemma~\ref{lemma:element} provides
 a $4$-pc~decomposition of  $H' \cup D$; both cases yield a contradiction.}

 Hence, we conclude that all 
components of $H$ are hanging-square graphs.  

Let $\{H_i\}_{i \in
  [k]}$ denote the set of all components of $H$.
We first assume  that for some $i \in [k]$, the graph $H_{i}$ has a
square $Q$. By assumption (Case 2), the deletion of $E(Q)$ yields at least $2$~components,
thus, $V(H_{i}) \cap V(C) \subset V_q$, where {$ V_q=
  V(Q)\setminus V(T_{H_i})$} and $H_i$ does not contain further squares.

Let $P_0$ be a brrb-path such that $|V(P_0) \cap V(Q)|$ is maximum;
  if $Q$ is a $2$-bunch we let $P_0$ be the brrb-path occupied by $Q$,
  otherwise, we just let $P_0$ be any of the brrb-paths that {contain} the joint 
  of $Q$.
Suppose that $T_{H_i}$ has a building sequence starting at $P_0$ and ending at $P$ such that $P_0 \neq P$.  
Since $V(H_{i}) \cap V(C) \subset V_q$, the graph $G-E(P)$ is connected and,
{because of our assumption, it} is not a hanging-square graph.  By the
minimality of $G$, the graph $G-E(P)$ has a $4$-pc~decomposition. But this yields
a $4$-pc~decomposition of $G$, a contradiction.
Therefore, $H_i$ is the union of $P_0$ and a square~$Q$
and $G-E(Q)$ has exactly two components, $P_0$
 and $\tilde{H}$ (with $C$ a cycle in $\tilde{H}$). 
 If $\tilde{H}$ is not a hanging-square graph,
 then we proceed as in the previous argument (with $K=\tilde{H}$) to obtain 
 a $4$-pc~decomposition of $G$ and thus, a contradiction. 
 Therefore, we can assume that $\tilde{H}$ is a hanging-square graph and $C$ is a square of $\tilde{H}$. 
 Using the same argument as before, $\tilde{H}$ is the union of a $3$-path
 $\tilde{P}$ and a square~$C$.
 By Lemma~\ref{lemma:2basis}, we have 
 that there is a $4$-pc~decomposition of $H_i \cup \tilde{H}$ and therefore of~$G$, a contradiction.

 \smallskip
 
Finally, {we are left with the} case that $H_i$ is a skeleton for each $i \in [k]$.  
Suppose that for some $i \in [k]$, $H_i$ is not a $3$-path.
 Let $P$ be the last path of a building sequence of $H_i$.  If $(H_i \cup C)
 - E(P)$ is connected, then $G-E(P)$ is connected and is not a
 hanging-square graph. Thus, $G-E(P)$ has a $4$-pc~decomposition, and
 so does $G$, a contradiction. Then, $(H_i \cup C) - E(P)$ is not
 connected for all such paths.  This implies that $V(C)\cap V(H_i)
 \subset V(P)\setminus\{v\}$, where $v$ is the joint of $P$. Now, let $v' \in
 V(C)\cap V(H_i)$ and $P_0, P_1, \ldots, P_t$ be a building sequence of
 $H_i$ such that $v' \in P_0$. Since $v' \notin V(P_t)\setminus \{v''\}$, where $v''$
 is the joint of $P_t$, we have that $(H_i \cup C) - E(P_t)$ is
 connected.  This contradicts the previous assertion.

{Thus, for every $i \in [k]$, $H_i$ is a $3$-path.}
 Note that the degree in $G$ of a vertex in $C$ is in $\{2,3,4\}$.
 {Let $x_{1}, x_{2}, \ldots x_{t}$ be the subsequence of vertices
   of (the sequence that defines) $C$ that have degrees in $\{3,4\}$.}
 Assume first that every component intersects $C$ in exactly one
 vertex, and denote by $H_{i}$ the $3$-path that contains vertex
 $x_{i}$.  Furthermore, denote by {$C(i,j)$ a path in $C$ with {end vertices}
   $x_i$, $x_j$}.  We need one more notation: we denote by $L_{i}$
 and $S_{i}$ the longest and the shortest, respectively, {subpaths} 
 in $H_{i}$ with end vertex $x_{i}$.  We decompose $G$ into paths~$P'_{i}, i
 \in [t]$, where $P'_{i}= L_{i} \cup C(i,i+1) \cup S_{{i+1}}$ for
 all $i \in [t-1]$ and $P_{t}= L_{t} \cup C(t,1) \cup S_{{1}}$.
From next observation follows that the length
 of each path $P_{i}, i \in [t]$, is at least~4.
\begin{observation}\label{obs:property}
If $v$ is a vertex of degree 3 in~$C$, then its neighbors in $C$ have degree 2, since otherwise
there exists an odd vertex within distance 2 of $v$.
\end{observation}

{It remains to analyze the case that there is a $brrb$-path
  $\tilde{H}$ that intersects $C$ in at least 2~vertices. If
$C$ contains two adjacent vertices of $\tilde{H}$, then there exists in $G$
a cycle shorter than~$C$, a contradiction.  If $C$ contains an
inner and an end vertex of $\tilde{H}$, then $\tilde{H}\cup C$ is a
hanging-square graph, because in this case the length of $C$ is
necessarily~4 (otherwise there would be a shorter cycle). Let $v$ be
a vertex of $C$ that has degree~2 in $\tilde{H}\cup C$ (since $C$ has length~4, there are two such vertices), 
by Observation~\ref{obs:property}
the vertex $v$ has degree~2 in $G$ as well. Then, $G=\tilde{H} \cup C$, a
contradiction.  If $C$ contains both ends of $\tilde{H}$, then necessarily
$C$ has length~6 and again by Observation~\ref{obs:property} we have that $G=\tilde{H}\cup
C$, a contradiction, as this implies that $G$ has a
$4$-pc~decomposition.  This completes the proof that $G$ is a tree.}
\end{proof}


\subsection{Hanging-square graphs do not have 4-pc decompositions}\label{subsec:no4pc}
In this section we complete the proof of Theorem~\ref{theo:main}. 
{
\begin{proposition}\label{prop:hsgrhno4}
 If $H$ is a hanging-square graph, then $H$ does not have a 4-pc~decomposition.
\end{proposition}}
\begin{proof} Let $H$ be a minimum counterexample to the statement.
Let $T_H$ denote the skeleton of~$H$ and $\mathcal{B}_i$ denote
the set of (maximal) $i$-bunches of squares of $H$ for $i \in \{1,2\}$ (see Definition~\ref{def:h-sss}).
Let $\mathcal{D}$ be a $4$-pc~decomposition of $H$.  Clearly, $|E(H)|>3$. 

We note that if $\mathcal{D}$ contains a cycle $C$, then $\mathcal{D}\setminus\{C\}$ is a 
$4$-pc~decomposition of $H-E(C)$, which is a hanging-square graph, a
    contradiction to the minimality of $H$.  Thus, we assume that
    $\mathcal{D}$ consists only of paths.  
    
Recall that by Proposition~\ref{prop:fromanypath}, for each brrb-path $P$ of $T_H$,
there exists a building sequence of $T_H$ that starts at $P$.
We first prove two claims:
\begin{claim}\label{claim:bbredge}
For every $e \in E(T_H)$,
either $H-e$ is connected or one of the components of $H-e$ has at most 2 edges.
\end{claim}\noindent
\emph{Proof of Claim~\ref{claim:bbredge}.}
By Observation~\ref{obs:1}, there exists a brrb-path $P$ which contains $e$.
By contradiction, suppose that the components $H'$, $H''$ of $H-e$ satisfy   
\begin{equation}\label{eq:prolema}
{|E(H')| \geq 3 \,\, \text{and} \,\,  |E(H'')| \geq 3.}
\end{equation}
We show that this situation yields a contradiction. We have that $H' \cup P$ and $H'' \cup P$
are hanging-square graphs with {fewer} than $|E(H)|$ edges. Then, by the minimality
of $H$, neither $H' \cup P$, nor $H'' \cup P$ has a $4$-pc~decomposition.
Let $P(e) \in \mathcal{D}$ denote the path that covers~$e$; recall that
$\mathcal{D}$ is a $4$-pc~decomposition of $H$ which consists only of paths. 
If $H' \cup P$ contains at least 4~edges from $P(e)$, then the restriction
of $\mathcal{D}$ to $H' \cup P$ is
a $4$-pc~decomposition of $H' \cup P$, a contradiction.
Then, we can assume that each of the hanging-square graphs $H' \cup P$ and $H'' \cup P$
have at most~3 edges from $P(e)$.
Let $P=x_0x_1x_2x_3$ and suppose that $e=x_0x_1$, $H'$ contains $x_0$ and $H''$ contains $x_1, x_2, x_3$.
{ If no edge of $P(e)$ belongs to $E(H')$, then $H'' \cup P$ contains
  all edges of $P(e)$, a contradiction. Therefore, at least one edge
  from $P(e)$ belongs to $E(H')$.  It implies that the restriction of
  $\mathcal{D}$ to $H'$ consists of a set of paths of length at least~4, 
  say $\mathcal{P}'$, and a path $Q$ of length at least 1 with end
  vertex $x_0$.  Therefore $\mathcal{P}' \cup \{Q\cup P\}$ is a
  $4$-pc~decomposition of $H' \cup P$, a contradiction to the minimality of $H$.  {So, we
    now assume} that $e=x_1x_2$, $H'$ contains $x_0$, $x_1$ and $H''$ contains $x_2,
  x_3$. We note that the condition $e=x_1x_2$ implies that $P$ is not occupied at $x_0,x_2$.
  Without loss of generality, suppose that $H'$ has at least 2 edges
  from $P(e)$. Then, the restriction of $\mathcal{D}$ to $H'$ consists
  of a set of paths of length at least~4, say $\mathcal{P}'$, and a
  path $Q$ of length at least~2 with end vertex~$x_1$.  Therefore
  $\mathcal{P}' \cup \{Q\cup \{x_1x_2, x_2x_3\}\}$ is a $4$-pc~decomposition 
  of $H' \cup P$, again a contradiction to the minimality of $H$. 
  } 
$\Box$

\medskip
%

\begin{claim}\label{claim:page13}
If $v \in V(T_H)$ is the joint of a building path or of a 1-bunch,
then every building path of~$H$ and every $1$-bunch of $H$ has joint $v$
and every $2$-bunch has $v$ as a joint. 
\end{claim}
\noindent
\emph{Proof of Claim~\ref{claim:page13}.}
If there are distinct $v, u \in V(T_H)$ which are joints of building paths or of 1-bunches,
then, by definition of hanging-square graphs, there exists an edge $e\in E(T_H)$ so that 
$H - e$ has two components, each with at least 3 edges, 
contradicting Claim~\ref{claim:bbredge}. Similarly, if $v$ is the joint of 
all building paths and 1-bunches and a 2-bunch of $H$ does not have $v$ as a joint, then again
due to the definition of hanging-square graphs, there exists an edge $e$ such that $H - e$ has two components
each with at least 3 edges, a contradiction to Claim~\ref{claim:bbredge}. $\Box$

\medskip

Let $P=x_0x_1x_2x_3$ be a brrb-path of $T_H$. If no vertex of $P$ is the joint of a building path or of a $1$-bunch,
then a longest path in $H$ has
    length at most~4 and hence, $\mathcal{D}$ consists of paths
    of length~4. But this implies that $|E(H)|$ is even, a
    contradiction.

Suppose $x_2$ is the joint of a building path or of a $1$-bunch. 
Let $Q \in \mathcal{D}$ be the path that contains the edge $x_2x_3$.
    We have that $|E(Q)|=4$, since the longest path in $H$ that
    contains $x_2x_3$ has length~4, due to Claim~\ref{claim:page13}. Further, $Q$ contains either edges of a 1-bunch
    or of a $2$-bunch, and in either case the deletion of $E(Q)$ generates a hanging-square graph,
    a contradiction to the minimality of $H$.
   
Finally, suppose  $x_0$ is the joint of a building path or of a $1$-bunch. 
Since $x_0$ is an odd degree vertex and due to Claim~\ref{claim:page13}, there exists a
    path $Q \in \mathcal{D}$ that ends at $x_0$, but the longest such path in $H$ has length~3, a contradiction. 
\end{proof}

%
%
%

\subsection{Proof of Corollary~\ref{coroll:main}}\label{sec:cor2}

%

{Let $P$ be a path and $x,y \in V(P)$. 
We denote by $P(x,y)$ the subpath in $P$ with {end vertices} $x, y$.
Let $C$ be a cycle and $x,y \in V(P)$, with $x\neq y$. 
We denote by $C(x,y)$ one of the two subpaths in $C$ with {end vertices} $x, y$.

We need the following lemma to complete the proof of Corollary~\ref{coroll:main}.

\begin{lemma}\label{JGTcycles}
Let $C$ be a cycle and $D$ be a path or a cycle.
Suppose $C$ and $D$ are of length at least~4 and at most~7.
If $V(C) \cap  V(D) {\neq} \emptyset$, 
$E(C) \cap  E(D)=\emptyset$ and $C \cup D$ is a triangle-free graph
which, in the case that $C$ is a cycle of length 4 and $D$ is a path, 
does not satisfy the following (see Figure~\ref{fig:forbgraphs}):
\begin{enumerate}
 \item[(i)]$|V(C) \cap  V(D)|=1$
 and $V(C) \cap  V(D)$ is an end vertex of $D$,
  \item[(ii)] $|V(C) \cap  V(D)|=2$, and
 $V(C) \cap  V(D)$ contains an end vertex $x$ of $D$ and a vertex at distance two of $x$ in $D$,
\end{enumerate}
then {$C\cup D$} has a decomposition into paths of length at least~4.
\end{lemma}
\begin{figure}[h]
 \centering
\ifpdf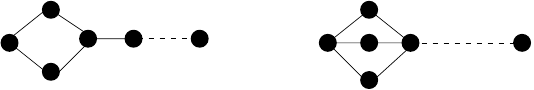\else\input{forbgraphs.ps_t}\fi 
   \caption{The two graphs of Lemma~\ref{JGTcycles}. Dashed lines represent paths.}\label{fig:forbgraphs}
 \end{figure}
\begin{proof}
Note that if $D$ and $C$ are cycles of length 4, then one can easily find a decomposition
into two paths of length 4. We now show that the case that $D$ is a cycle follows
from the case that $D$ is a path. Suppose that $D$ is a cycle of length at least 5.
We claim that $|V(C) \cap V(D)|{\leq}4$. Otherwise, 
there is a path $v_1v_2v_3$ in $C$ such that $v_i \in V(C) \cap V(D)$ for each $i\in\{1,2,3\}$.
Since  $C\cup D$ is triangle-free, each path $D(v_1,v_2)$ and $D(v_2,v_3)$
is of length at least~3, and due to that the length of $D$ is at most 7,
we have that $v_1v_3$ is an edge of $D$, and then, $\{v_1, v_2, v_3\}$ is a triangle of $C \cup D$, a contradiction.
Therefore, there is a vertex $x$ in $D$ of degree two in $C \cup D$. 
Let $e=xy$ be an edge in $D$. Then the path $D':= D-e$ is of length at least 4.
If the graph $C \cup D'$ satisfies condition$~(i)$ or$~(ii)$, then one can easily find the desired
decomposition of $C \cup D$. If $C \cup D'$ does not satisfy$~(i)$ and$~(ii)$,
under the assumption of the validity of the lemma for the cycle path case,
we show that we can find the desired decomposition of $C \cup D$.
Note that vertices $y$ and $x$ are the only odd vertices in $C \cup (D-e)$ and $x$ has degree~1 in $C \cup (D-e)$. 
Let $\mathcal{D}$ be a decomposition of  $C \cup D'$ into paths of length at least~4.
If a path in $\mathcal{D}$ ending in~$y$ does not use the vertex $x$, then we can easily extend $\mathcal{D}$ 
to the desired decomposition of $C \cup D$. 
Now suppose that each path in $\mathcal{D}$ ending in~$y$ uses the vertex $x$ and let $P$ be such path.
We have that $P \cup e$ is a cycle of length at least~5 and that $C \cup D' - E(P)$ is a cycle of length 8 or 9. 
Moreover, $\mathcal{D}\setminus P$ is a decomposition of $C \cup D' - E(P)$ into paths of length 4 or 5. 
Let $P'$ be a path in $\mathcal{D}\setminus P$ that contains vertex $y$. Then, via the assumption there is a  
decomposition $\mathcal{D}'$ into paths of length at least $4$ of the graph $P \cup e \cup P'$. Hence, taking the union 
of $\mathcal{D}'$ and the path in $\mathcal{D}\setminus \{P,P'\}$ we have a decomposition of $C \cup D$
into paths of length at least~$4$.

We assume now that $D$ is a path. In the case that $|V(C)\cap V(D)|=1$, it is not hard to see that,
unless $C\cup D$ satisfies condition~$(i)$, we can obtain the desired path
decomposition. Set $V(C)\cap V(D)=\{u\}$. One of the paths 
of the desired path decomposition of $C\cup D$ is obtained by the union of a longest
subpath $P$ of $D$ with end vertex $u$ and a subpath $P'$ of $C$ with $u$ as an end vertex: if the length
of $P$ is at least~$3$, then take $P'$ consisting of only one edge; and if not, consider $P'$
such that $\ell(P')=\ell(P)=2$; in both cases, since $\ell(P)<\ell(D)$ because
$P\cup D$ does not satisfy~$(i)$, the complement of $P\cup P'$ is a path of length at least~$4$.
Let $V(C)\cap V(D)=\{u,v\}$, let $x,y$ be the end vertices of $D$ such that
$D(x,u)$ does not contain $v$ and let $P_1, P_2, P_3$ be the three subpaths
of $C\cup D$ with end vertices $u, v$.
Without loss of generality, suppose $\ell(D(x,u)) \geq \ell(D(y,v))$ and $\ell(P_1)\leq \ell(P_2)\leq \ell(P_3)$. 
One of the paths of the desired path decomposition of $C\cup D$ is obtained by the union of 
$D(x,u)$, $P_1$ and a subpath $P'$ of $P_3$ with end vertex $v$; the other path is its complement.
Let us check that it is possible to choose a suitable $P'$. 
If $\ell(P_2)\geq 3$ and $\ell(P_3) > 3$, then $P'$ can be chosen so that 
$\ell(P')=3$, and the result follows. If $\ell(P_2)=\ell(P_3)=3$, then
$\ell(D(x,u))\geq 1$, and hence $\ell(D(x,u)\cup P_1)\geq 2$, and the result holds if we choose
$P'$ so that $\ell(P')=2$. 
If $\ell(P_2)= 2 $ and $ \ell(P_3) \geq 3$, then $\ell(P_1) = 2$ and 
again the result holds if we choose
$P'$ so that $\ell(P')=2$. 
Finally, if $\ell(P_2)= \ell(P_3) = 2$, then $\ell(P_1) = 2$ 
and $\ell(D(x,u)) \geq \ell(D(y,v)) \geq 1$
since $C\cup D$ does not satisfy~$(ii)$, and the result follows
if we choose $P'$ consisting of one edge.


For the cases that $|V(C)\cap V(D)|= i$ with $i \geq 3$,
let $V(C)\cap V(D)=\{v_0,\ldots,v_{i-1}\}$ and $x,y$ be the end vertices of $D$ such that
$D:= x \cdots v_0 \cdots v_2\cdots \cdots v_{i-1}\cdots y$. 

If there is a vertex $w \in V(C)$ of degree 2 in $C\cup D$ adjacent to $v_0$ (symmetrically for $v_{i-1}$) 
then, the graphs $wv_0 \cup D(v_0,y)$ and $(C-wv_0) \cup D(x,v_0)$ form a path decomposition of $C\cup D$. 
These paths are of length at least 4 unless $i=3$ and
$\ell(D(v_0,y))=2$ or $\ell(C)=4$ and $x=v_0$.  If $\ell(D(v_0,y))=2$ (resp. $\ell(C)=4$ and $x=v_0$),
then, since $C\cup D$ is triangle-free there is a vertex $w' \in V(C)$ such that
$C(v_0,w') \cup D(v_0,y)$ and $(C-C(v_0,w')) \cup D(x,v_0)$ 
(resp. $D(w',v_0)\cup C(v_0,v_2) \cup D(v_2,y)$ 
and $(C-C(v_0,v_2)) \cup D(w',v_2)$) form a desired path decomposition.
Thus, we can assume that
\begin{equation}\label{eq:cond}
\text{there is no vertex in $V(C)$ of degree 2 in $C\cup D$ adjacent to $v_0$ or to $v_{i-1}$.}  
\end{equation}

 Suppose that $v_0$ and $v_1$ are neighbors in $C$ (analogously, $v_{i-1}$ and $v_{i-2}$ are neighbors in $C$ ).
Then the path $D(v_0,v_1)$ is of length at least 3 because $C \cup D$ is triangle-free. Let $w$ be the neighbour of $v_1$ in $D(v_0,v_1)$.
Thus, we have that $wv_1 \cup (C - v_0v_1) \cup D(v_0,x)$ and its complement with respect to $C\cup D$ is a desired path decomposition,
and we can assume that
\begin{equation}\label{eq:cond1}
\text{$v_0$ and $v_1$  (resp. $v_{i-1}$ and $v_{i-2}$) are not neighbors in $C$}  
\end{equation}
%
Suppose now that there are $j \in \{0,\ldots,i-2\}$ and $v \in V(C \cup D)$ of degree two in $C\cup D$ such that
the following two conditions hold (see Figure~\ref{fig:conf1}(a)):
\begin{itemize}
 \item  $v_0$ and $v_{j+1}$ are neighbors in $C$ (resp.  $v_{i-1}$ and $v_j$ are neighbors in $C$)

    \item $v$ is an inner vertex of $D(v_j,v_{j+1})$ or it is adjacent to $v_j$  (resp. $v_{j+1}$) in the 
 path $C(v_j,v_{j+1})$,
 where $C(v_j,v_{j+1})$ is the path of $C$ connecting $v_j$ and $v_{j+1}$ which
 does not contain $v_0$ (resp. $v_{i-1}$).
\end{itemize}

 \begin{figure}[h]
 \centering
\scriptsize{
 \subfigure[]
  { \ifpdf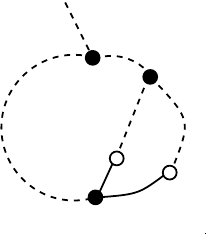\else\input{conf1.ps_t}\fi } \qquad
 \subfigure[]
  {\ifpdf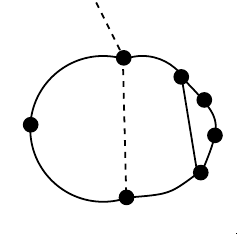\else\input{conf2.ps_t}\fi }\qquad
   \caption{In (a), vertex $v$ is one of the white vertices. Dashed lines represent paths.}
 \label{fig:conf1}
}
 \end{figure}

In this case, there exists a desired path-decomposition of $C\cup D$.
Because of assumption~(\ref{eq:cond1}), $j \neq 0$ (resp.  $j \neq i-2$).
Let $C'$ denote the path in~$C$, disjoint from $C(v_j,v_{j+1})$, joining $v_0$ and $v_{j}$ (resp. $v_{j+1}$ and $v_{i-1}$).
If $v$ is an inner vertex of $D(v_j,v_{j+1})$ then, one can easily verify that the graph 
$D(v,v_{j+1}) \cup C(v_j,v_{j+1}) \cup C' \cup D(v_0,x)$
(resp. $D(v,v_{j})\cup C(v_j,v_{j+1}) \cup C'\cup D(v_{i-1},y)$)
and its complement with respect to $C\cup D$
are paths of length at least four. Similarly, if 
$v$ is adjacent to $v_j$ (resp. $v_{j+1}$) as stated in the second condition, the graph 
$(C(v_j,v_{j+1})-v_jv) \cup D(v_j,v_{j+1}) \cup  C' \cup D(v_0,x)$
(resp. $(C(v_j,v_{j+1})-v_jv) \cup D(v_j,v_{j+1}) \cup  C' \cup D(v_0,x)$)
and its complement with respect to $C\cup D$ form the desired path decomposition.
Hence, we can assume that 
\begin{equation}\label{eq:cond2}
\text{ such pair $j$, $v$ does not exist.}  
\end{equation}

We now finish the proof. Let $C_1$ and $C_2$ denote the two paths joining $v_0$ and $v_1$ in $C$.
Due to assumption~(\ref{eq:cond1}), there are $l,k \neq 1$ in $\{2,\ldots, i-1\}$ such that $v_l$ is
an inner vertex of $C_1$ and $v_k$ is
an inner vertex of $C_2$. Because of assumption~(\ref{eq:cond}), we can assume that
$v_lv_0$ and $v_0v_k$ are edges of $C$. Hence, due to assumption~(\ref{eq:cond2}) and the fact that
$C\cup D$ is triangle-free, we have $|l-k|>1$, and $D(v_k,v_{k-1})$, $D(v_l,v_{l-1})$ are edges. 
Without loss of generality, we can assume $l<k$.
If $v_{k-1}$ is an inner vertex of $C_2(v_1,v_k)$, then there are two more inner vertices $w, w'$,
since $C\cup D$ is triangle free and because of assumption~(\ref{eq:cond2}) one of these vertices, the neighbour
of $v_{k-1}$ in $C$ (say $w$), is in $|V(C)\cap V(D)|$ (see Figure~\ref{fig:conf1}(b)). Indeed, since $\ell(C)\leq 7$,  the path $C_2$ is $v_0v_kw'wv_{k-1}v_1$.
If $v_k$ were $v_{i-1}$, then because of assumption~(\ref{eq:cond2}), $D(v_0,v_1)$ is an edge
and since $\ell(C)\leq 7$, the vertices $v_0,v_1,v_l$ would create a triangle, a contradiction.
In fact, $D(v_0,v_1)$ is a path of length at least 2. 
Thus $v_{i-1}$ is $w$ or $w'$.
In either case, because of assumptions~(\ref{eq:cond}) and~(\ref{eq:cond2}), both $w, w'$ are in $V(C)\cap V(D)$,
and the vertex $v_k$ is connected to $w$ or to $w'$ throughout edges of $D$, which is not possible without
creating a triangle.
Now, assume that $v_{k-1}$ is an inner vertex of $C_1-v_lv_0$.
As before, if $v_k$ were $v_{i-1}$, then because of assumption~(\ref{eq:cond2}), $D(v_0,v_1)$ would be an edge.
Since $C\cup D$ is triangle free and because of assumption~(\ref{eq:cond}) there is a vertex $v_p$ with $p\neq 0, 1$
neighbor of $v_k$ in $C$, and either  $p<l$ or $l<p$. If $p<l$, then, since $C\cup D$ is triangle free $\ell(C)+\ell(D)\geq 15$,
a contradiction. Assume now that $l<p$. Because of assumption~(\ref{eq:cond2})
and since $C\cup D$ is triangle-free, there exists a vertex $v_t \neq v_1$ such that $v_t \in V(C)\cap V(D)$ 
is an inner vertex of $D(v_1,v_l)$ and $v_tv_l$ is an edge. But this implies that $\ell(C)+\ell(D)\geq 15$,
again a contradiction. Finally, if $v_k \neq v_{i-1}$, then $v_{i-1}$ is an inner vertex of $C_1(v_l,v_{k-1})$,
$C_1(v_1,v_{k-1})$, or $C_2(v_1,v_k)$. It can be checked that any case leads to a contradiction.
%
\end{proof}}

\medskip
{
\begin{proof}[Proof of Corollary~\ref{coroll:main}]
 {Let $G$ be a graph in $\mathcal{G}$ on $n$ vertices  with $|E(G)|\leq 4\,\lceil{n}/{2}\rceil$.}  
If $G$ is a hanging-square graph, then we use Proposition~\ref{obs:gallaiforhs} to conclude the result.
If $G$ is a cycle, then we can decompose $G$ into two paths and the statement holds.   
Otherwise, by Theorem~\ref{theo:main}, $G$ admits a $4$-pc~decomposition.
Let us consider a $4$-pc decomposition $\mathcal{D}$ of $G$ 
  which is maximal with respect to the number of paths. Clearly, $|\mathcal{D}| \leq \lceil{n}/{2}\rceil$.
 Note that since $\mathcal{D}$ is maximal with respect to the number of paths,
 every element in $\mathcal{D}$ has length at most 7.
If $\mathcal{D}$ contains paths only, then the corollary holds.
If not, suppose that $\mathcal{D}$ contains a cycle.
Because of Lemma~\ref{JGTcycles}, every cycle in $\mathcal{D}$
is of length~4 and if $D$ is an element of $\mathcal{D}$ 
such that $V(C)\cap V(D) \neq \emptyset$ and $E(C)\cap E(D){=}\emptyset$, then
$D$ is a path and the graph $C \cup D$ is one of the two graphs described in~Lemma~\ref{JGTcycles} (see Figure~\ref{fig:forbgraphs}). 
Moreover, since $G$ is connected and is not a cycle, such $D$ exists for each cycle in $\mathcal{D}$.
Let $C_1, C_2, \ldots, C_k$ be the cycles in $\mathcal{D}$
and $D_1, D_2, \ldots, D_k$ be paths (not necessarily distinct) such that 
$V(C_i)\cap V(D_i) \neq \emptyset$ and $E(C_i)\cap E(D_i){=}\emptyset$ for each $i\in\{1,\ldots,k\}$. 
Because of the structure of the graphs $C_i\cup D_i$ (Figure~\ref{fig:forbgraphs}), 
for each $i$ there is at most one $j\neq i$ such that $D_i=D_j$.
Thus, in order to complete the statement, it suffices to show that for each such graph $C_i\cup D_i$ or
$C_i\cup D_i \cup C_j$ with $D_i=D_j$ there is a decomposition into 2 or 3 paths, respectively, which is a fairly trivial task.
\end{proof}}

\section{Properties of hanging-square graphs}\label{sec:proph-sgraphs}

This section addresses properties of hanging-square graphs. {These
properties were used in the proof of the main results, but some of
them are interesting in their own right.}
Recall that if $H$ is a hanging-square graph, then $T_H$ is its skeleton
(see Definitions~\ref{def:ske}  and~\ref{def:h-sss}).


\begin{lemma}~\label{lemma:length5} Let $H$ be a hanging-square graph and $C$ be
  a cycle of length at least~$5$ such that $E(H)\cap E(C)=\emptyset$ and
  $V(T_H)\cap V(C)\neq \emptyset$.  If $H\cup C \in \mathcal{G}$, then $H\cup C$ has a 4-pc decomposition.
\end{lemma}
\begin{proof}
  Let $\lambda$ be the coloring of $T_H$ and $v \in V(T_H) \cap V(C) $.  If
  $\lambda(v)$ is black, then by Observation~\ref{obs:1} there exists a brrb-path
  $P$ such that $v$ is an end vertex of $P$.
  Moreover, by Proposition~\ref{prop:fromanypath}, $P$ is the starting
  path of some building sequence of~$T_H$.  Hence, it suffices to prove that
 {$P\cup C$} has a $4$-pc~decomposition. Let $v'$
  be a neighbor of $v$ in $C$. Then, $v' \notin V(P)$, because $d_o(P\cup
  C)\geq 3$ and $P\cup C$ is triangle-free, and thus,  
  $P\cup vv',$ $C-vv'$ are paths of length at least~4 that decompose $P \cup C$.  We now suppose that $\lambda(v)$ is red. Again
  by Observation~\ref{obs:1} and Proposition~\ref{prop:fromanypath}, 
  there exists a brrb-path~$P$ containing~$v$ that is the starting path of some
  building sequence of~$T_H$.  In this case, we have that $v$ is an inner vertex
  of $P$.  We can assume that $C$ does not intersect the ends of $P$, otherwise
  we use the previous case to complete the proof.  We decompose {$P\cup C$}
  into two paths of length at least~$4$ in the following way. Let $vx$
  be the edge of $P$ such that $\lambda(x)$ is black and let $vv'u$ be a path of
  length~$2$ in $C$. Given that $C$ does not intersect the ends of $P$ and that
  $P\cup C$ is triangle-free, both $(P-xv) \cup vv'u$ and $(C-vv'u)
  \cup xv$ are paths of length at least~$4$.
\end{proof}

\begin{lemma}~\label{lemma:length4} 
  Let $H$ be a hanging-square graph, $Q$ be a square of $H$, and $C$ be a cycle
  of length at least~4 such that $E(H)\cap E(C)=\emptyset$, $V_q \cap
  V(C)\neq \emptyset$, where {$V_q = V(Q) \setminus V(T_H)$,} and $H\cup C$ is not a hanging-square graph. If $H\cup C \in \mathcal{G}$, then $H\cup C$ has a 4-pc~decomposition.
\end{lemma}
\begin{proof}
  Let $\lambda$ be the coloring of $T_H$. We split the proof into two cases. \bigskip

  \noindent
  \textbf{Case $V(Q) \cap V(T_H)=\{v\}$.}
  Let $Q(v,v')$ be a shortest path in $Q$
  {connecting $v$ and $v' \in V_q \cap V(C)$.}  As
  before, by Observation~\ref{obs:1} and Proposition~\ref{prop:fromanypath}, 
  if $\lambda(v)$ is black (resp. red), then there exists a
  brrb-path $P$ that has $v$ as an end
    (resp. inner) vertex and $P$ is the starting
    path of some building sequence of $T_H$.  Hence, it suffices to
  show that we can find a $4$-pc~decomposition of {$P\cup Q \cup
    C$.}  
    
  Firstly, we suppose that $V(C)\cap V(P)=\emptyset$. 
  It is easy to see that, we can decompose $P\cup Q \cup C$ into two
paths of length at least $4$: form the first path $P'$ by starting at an arbitrary end of
$P$, continue by picking up the edges of $Q-Q(v, v')$ and add one more edge incident
with $v'$ in $C$. This is a path of length at least 4, and it is easy to check that also
the remaining edges form a path of length at least 4.
%

  
  Secondly,
  we suppose that $V(C)\cap V(P)\neq\emptyset$. If $\ell(C)\geq5$, then the result holds by Lemma~\ref{lemma:length5}.  We
  assume that $C$ has length 4. Let $P=x_0x_1x_2x_3$ (we can assume that
  either $v=x_0$, or $v=x_2$).  Moreover, $|V_q \cap V(C)|\leq2$. If
  $|V_q \cap V(C)|=\{u,w\}$, with $u \neq w$, then necessarily $u$ and
  $w$ are adjacent to $v$ and $|V(C)\cap V(P)|=1$. Let $V(C)\cap
  V(P)=\{z\}$. If $\lambda(v)$ is black, then $z=x_2$.  If
  $\lambda(v)$ is red, then $z=x_0$. In the first case we decompose
  {$P\cup Q \cup C$} into two paths, with one of them defined by
  $vuzwu',$  where $u' \in V_q$.  In the second case we decompose $P\cup Q
  \cup C$ into two paths, with one of them defined by $x_0ux_2wu'$, where $u' \in V_q$.  
  We now study the case that $|V_q \cap
  V(C)|=\{u\}$.  If $u$ is adjacent to $v$, then $v \notin V(C)\cap
  V(P)$. Moreover, $|V(C)\cap V(P)|=1$. Let $V(C)\cap V(P)=\{z\}$.
  For each choice of $v$ in $\{x_0, x_2\}$ we have that $P\cup Q
  \cup C$ can be decomposed into a cycle and a path, with the cycle
  defined by $uv \cup C(u,z) \cup P(z,v)$, where $uv \in E(Q)$ and
  $C(u,z)$ is a shortest path in $C$ (with end vertices~$u$ and~$z$).  
  
  We are left with the case that $u$ is not adjacent to $v$.  It is clear that $|V(C)\cap V(P)|\leq2$.
Moreover, we have that $v \notin V(C)\cap V(P)$, otherwise $H \cup C$ would be a hanging-square graph. We first
  suppose that $|V(C)\cap V(P)|=2$.  
  Let
  $Q(u,v)$ be
  a path in $Q$ with end vertices $u$, $v$, and $C(x_1,x_3)$ be a
  path in $C$ with end vertices $x_1, x_3$.
  We have that for either choice of $v$, $V(C)\cap V(P)=\{x_1,x_3\}$, 
  and we can decompose {$P\cup Q \cup C$} into a cycle and a path, where 
  the cycle is  defined by $Q(u,v)
  \cup ux_3 \cup C(x_1,x_3) \cup vx_1$, in the case that $v=x_0$, and by $x_1u \cup
  Q(u,v) \cup vx_3 \cup C(x_1,x_3)$, in the case that $v=x_2$. 
  Finally, let $V(C)\cap
  V(P)=\{z\}$. For either choice of $v$, we can decompose {$P\cup
    Q\cup C$} into a cycle and a path, where the cycle is defined by
  $Q(u,v) \cup C(u,z) \cup P(v,z)$, and $C(u,z)$ is a shortest path in $C$
  (with end vertices~$u$,~$z$).\bigskip
    
    \noindent \textbf{Case $V(Q) \cap V(T_H)=\{v,u\}$, $v \neq u$.}
    By Definition~\ref{def:h-sss} and Observation~\ref{obs:1},
    there exists a brrb-path~$P$
with $v$ as an end vertex, 
$u$ an innerl vertex and such that $v$ and $u$ are non-adjacent
(in other words, $P$ is occupied at $v$, $u$).
By Proposition~\ref{prop:fromanypath},
$P$ is the
starting path of some building sequence of $T_H$. 
As before, it suffices to show that $P\cup Q
\cup C$ has a $4$-pc~decomposition.  
Suppose that $V(C)\cap V(P)\neq\emptyset$. If the length of $C$ is at least 5, then we use Lemma~\ref{lemma:length5} to conclude
the statement. Thus, the length of $C$ is 4, and it is easy to see that $|V(C)\cap V(P)|>1$ is not possible.
Then the result follows from the previous case by interchanging the roles of $C$ and $Q$. 
Hence, we can assume that $V(C)\cap V(P)=\emptyset$.
If $V_q \cap V(C)=\{v'\}$, then  $P\cup vv' \cup vu'$, $(C-v'u')\cup (Q-vv')$, where $u'$ is a
neighbor of~$v'$ in~$C$, is a $4$-pc~decomposition of $P\cup Q
\cup C$.  
Assume $V_q \cap V(C)=\{v',\tilde{v}\}$, $v'\neq\tilde{v}$.
Let $x$ denote the black neighbor of $u$ in $P$ and
$C(v',\tilde{v})$ be a path in $C$ with end vertices $v'$, $\tilde{v}$. 
We can
decompose {$P\cup Q \cup C$} into a path  
given by $xu \cup uv' \cup C(v',\tilde{v}) 
\cup \tilde{v}v$ of length at least~4 and a cycle 
of length~4.
\end{proof}

\begin{lemma}\label{lemma:main}
  Let $T$ be a skeleton and $P$ be the last path in some building sequence
  of $T$.  Moreover, let $C$ be a square such that $E(T) \cap
  E(C)=\emptyset$, $V(T) \cap V(C)\neq \emptyset$ and $T \cup C \in \mathcal{G}$ is not a hanging-square graph.  In
  the following two cases $T \cup C$ has a 4-pc decomposition.
\begin{enumerate}
 \item  $[T-E(P)]\cup C$ is a hanging-square graph.
 \item  $[T{-}E(P)]\cup C$ consists of two connected components, namely $T{-}E(P)$ and $C$.
\end{enumerate}
\end{lemma}
\begin{proof}Let $\lambda$ be the coloring of $T$ and $v$ be the joint of $P$.  
We analyze {cases}~\emph{1} and~\emph{2} separately.

 \noindent {\emph{Case 1.}} Assume that $[T{-}E(P)]\cup C$ is a hanging-square graph. 
 We can assume that there exists a brrb-path $P_0=x_0x_1x_2x_3$ such that $E(P_0) \cap E(P)
 =\emptyset$ and $P$ is the last path of a building sequence starting at $P_0$. Moreover, we
 can assume that if $C$ is a 1-bunch with joint $u$, then
 $V(P_0) \cap V(C) = \{u\}$ and $u \in \{x_0,x_2\}$. Otherwise, $P_0$ is occupied at $x_0,x_2$ by $C$ in $[T-E(P)]\cup C$.

Suppose that  $[V(P)\setminus\{v\}] \cap V(C) = \emptyset$. Since $T\cup C$ is not a
    hanging-square graph, we have that $C$ is a 2-bunch and that either $P$ is a 4-path with joint
 $x_1$, or there exists another building path $P' \neq P$ of $T$ with
 joint $v' \in \{x_0,x_2\}$ such that $v \neq v'$ and $v\in\{x_0,x_2\}$. 
 In the first case
 we can easily see that $P_0\cup P \cup C$ can be decomposed into two
 paths of length at least~4. In the second case, {it is routine} to
 check that $P_0 \cup P \cup C \cup P'$ can be decomposed into $3$~paths
 of length at least~$4$.  
 We now suppose that $[V(P)\setminus\{v\}] \cap V(C) \neq \emptyset$.  Let us
 suppose that $C$ is a 1-bunch.  
 If $v=u$, then $T\cup C$ would be a hanging square graph, a contradiction.  
 All the remaining cases are worked out in
 Figure~\ref{fig:caselongleg}.  Hence, it remains to analyze the case
 that $P_0$ is occupied at $x_0,x_2$ by $C$. But in this case, $P$ would intersect $V_c =
 V(C)\setminus V(P_0)$, which yields a contradiction to $d_o(T\cup C)\geq 3$. Hence, the result follows.
\begin{figure}
 \centering
\ifpdf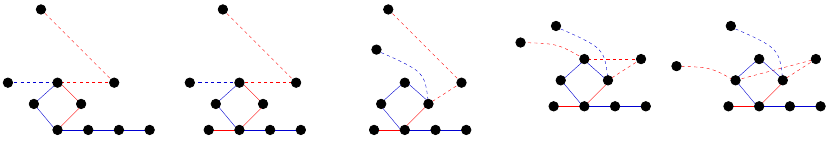\else\input{path_leg_cycle.ps_t}\fi
\caption{Proof of Lemma~\ref{lemma:main}, Case 1.  In all figures, {possibly}
  $v \in \{x_1, x_2, x_3\}$. Recall that $P_0=x_0x_1x_2x_3$ and $P = P' \sqcup
  \tilde{P}$. In the figures, $P'$ has ends $x'_0$, $v$ and $\tilde{P}$ has ends
  $\tilde{x}_0$, $v$.  Moreover, the length of {$P'$} and $\tilde{P}$ is either 2
  or 3 depending on whether $P$ is a 4-path or a 6-path.}
\label{fig:caselongleg}
\end{figure}
 
\medskip

\noindent {\emph{Case 2.}} Assume that $[T{-}E(P)]\cup C$ consists of exactly two connected components, namely $T{-}E(P)$ and $C$.
We have that all inner vertices of $P$ but its joint $v$ have degree~$2$ in $T$, 
the ends of~$P$ have degree~$1$ in $T$ and
$V(C)\cap V(T) \subset V(P)\setminus\{v\}$.  If $P$ is a 6-path (resp. 4-path), 
then $\lambda(v)$ is black (resp. red) and there exists a brrb-path $P_0$ that
is the starting path for some building sequence of $T - E(P)$ and such that $v$ is an
end (resp. inner) vertex of $P_0$.  In both cases, it suffices to show that $P_0\cup P\cup
C$ has a 4-$pc$ decomposition.  Let $P= P'\sqcup \tilde{P}$. 
If $C$ intersects only one of the paths $P'-v$, $\tilde{P}-v$, then it is
easy to see that $T \cup C$ is a hanging-square graph. Hence, $C$ intersects
both $P'-v$ and $\tilde{P}-v$. 
Let $V(C)\cap V(P')=x$, $V(C)\cap V(\tilde{P})=y$,
$x'$ (resp. $\tilde{x}$) be the end vertex of $P'$ (resp. $\tilde{P}$) distinct {from} $v$.

We suppose now that $P$ is a $6$-path.
Without loss of generality we can assume that $ \ell(P'(x',x)) \leq \ell(\tilde{P}(\tilde{x},y))$.
Let $C(x,y)$ denote the shortest path in $C$ connecting $x$ and $y$. Then, the following two paths
form a $4$-pc~decomposition of $P_0 \cup P \cup C$,
$$P'(x',x) \cup C(x,y) \cup \tilde{P}(y,v) \cup P_0 \quad {\rm and} \quad \tilde{P}(\tilde{x},y) \cup [C-C(x,y)] \cup P'(x,v).$$
The case that $P$ is a $4$-path is shown in Figure~\ref{fig:short_legsproof}.
\begin{figure}[h]
 \centering
  \ifpdf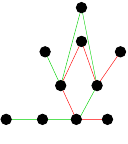\else\input{shortleg_1.ps_t}\fi\qquad
  \ifpdf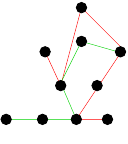\else\input{shortleg_2.ps_t}\fi\qquad
  \caption{All configurations in which $P$ is a $4$-path and \emph{Case~2} of
    Lemma~\ref{lemma:main} occurs.}
\label{fig:short_legsproof}
\end{figure}
\end{proof}

\begin{lemma}\label{lemma:element}
Let $H$ be a hanging-square graph, $Q$ be a square of $H$
and $V_q=V(Q) \setminus V(T_{H})$. 
Let $P'$ be a path such that $4 \leq\ell(P')\leq7$,  $E(P')\cap E(H)=\emptyset$
and $V(P') \cap V(H) \subset V_q$.
If $H \cup P'$ is triangle-free, then 
$H \cup P'$ has a $4$-pc~decomposition.
\end{lemma}

\begin{proof} 
  First of all, by Observation~\ref{obs:1}, Proposition~\ref{prop:fromanypath}
  and symmetry,
  we can assume that there is a brrb-path $P=x_0x_1x_2x_3	$ that is
  the starting path of some building sequence of $T_H$ such that if $Q$ is a
  $1$-bunch (resp. $2$-bunch) with joint $u$, then $V(P) \cap V(Q) =
  \{u\} \in \{x_0,x_1\}$ (resp.  $P$ is occupied at $x_0, x_2$ by $Q$ and $u\in\{x_0,x_2\}$). 
    Let $P' = u_1 \cdots u_{l}$, $Q= y_0\cdots y_3$ with $y_0=u$ and $Q(x,y)$ denote a shortest path in $Q$ connecting $x$ and $y$. 
  
  As before, to prove that $H\cup P'$ has a $4$-pc~decomposition,
  it suffices to prove that $P \cup Q \cup P'$ has a  $4$-pc~decomposition.
   We study cases according to the three possible values of $|V(P') \cap V_q|$.

   Let $V(P') \cap V_q=\{y_1,y_2,y_3\}$ and $y_{i}=u_{\pi(i)}$ for each $i$. In this case, $Q$ must be a 1-bunch.  
   By symmetry,
   either $\pi(1)<\pi(2)<\pi(3)$ or $\pi(2)<\pi(1)<\pi(3)$. In the first case, 
   the path  $P(x_3,u) \cup (Q- uy_1) \cup P'(y_1,u_1)$ and its complement
   with respect to $P \cup Q \cup P'$ and in the second case, 
   the path  $P(x_3,u) \cup uy_3 \cup P'(y_3,y_1) \cup y_1y_2 \cup P'(y_2,u_1)$ and its complement
   also with respect to $P \cup Q \cup P'$  form a $4$-pc~decompositions of $P \cup Q \cup P'$.

    Let $V(P') \cap V_q=\{w\}$. If $Q$ is a 2-bunch set $u=x_0$.
  Without loss of generality, we assume that $\ell(P'(u_1,w)) \leq \ell(P'(w,u_{l}))$.
  If $u\neq x_0$, then $P(x_3,u) \cup (Q-Q(u,w)) \cup P'(u_1,w) \quad \text{and} \quad
  P(x_0,u) \cup Q(u,w) \cup P'(w,u_{l})$
  form a 4-$pc$~decomposition of $P \cup Q \cup P'$, and  
 if $u= x_0$, then $P(x_3,u) \cup Q(u,w) \cup P'(u_1,w) \quad \text{and} \quad
  (Q-Q(u,w)) \cup P'(w,u_{l})$ 
  form a 4-$pc$~decomposition of $P \cup Q \cup P'$.


  Let $V(P') \cap V_q=\{u_i, u_j\}$ with $i<j$. If $Q$ is a 2-bunch set $u=x_2$.
  Without loss of generality we assume that $\ell(P'(u_j,u_l))\leq \ell(P'(u_i,u_1))$.
  If $u\neq x_0$ or $u= x_0$ then, $P(x_3,u) \cup (Q-Q(u,u_j)) \cup P'(u_j,u_l)$ and its complement with respect
  to $P \cup Q \cup P'$ form a 4-$pc$~decomposition of $P \cup Q \cup P'$. 
  \end{proof}

 \begin{lemma}\label{lemma:2basis}
   Let $H$, $H'$ be hanging-square graphs composed (only) of squares $Q$, $Q'$ and 
   brrb-paths $P$, $P'$, respectively.  Let  {$V_q = V(Q)\setminus
     V(T_H)$} and $V_{q'} = V(Q') \setminus V(T_{H'})$.
   Assume that $E(H)\cap E(H')=\emptyset$ and $V(H) \cap
   V(H') \subset V_q \cap V_{q'}\neq \emptyset$.  If $H \cup {H'} \in \mathcal{G}$, then it has a \hbox{4-pc~decomposition.}
 \end{lemma}
 \begin{proof}
Note that $|V(Q) \cap V(Q')|\in \{1,2\}$, otherwise $Q\cup Q'$ would contain a triangle.
Let $P=x_0 \cdots x_3$, $P'=x'_0 \cdots x'_3$ and $Q(x,y)$ (resp. $Q'(x,y)$) be a longest path 
in $Q$ (resp. $Q'$) connecting 
 $x$ and $y$. In addition, observe that 
 if $Q$ and $Q'$ were $2$-bunches, then there would exist a path of length 2 connecting
 $x_0$ and $x'_0$, a contradiction to $d_o(H\cup H')\geq 3$. Hence, at least one of $Q$, $Q'$
 is a $1$-bunch.

We first suppose that $V(Q) \cap V(Q')= \{v\}$. 
 Suppose that $Q,$ $Q'$ are $1$-bunches and let $u, u'$ be their joints, respectively.
 Without loss of generality $u\in\{x_0,x_1\}$ and $u'\in\{x'_0,x'_1\}$.
 In this case, a $4$-pc~decomposition of $H \cup H'$
 is formed of two paths, with one of them, the following: 
$ P(x_0,u) \cup Q(u,v) \cup Q'(u',v) \cup P'(x'_0,u').$
 We now assume that $Q'$ is a $2$-bunch with joints $x'_0,x'_2$
 and $Q$ is the $1$-bunch previously described. We have that the following paths belongs to
 the desired  $4$-pc~decomposition of $H \cup H'$ into two paths: $Q'(x'_0,v) \cup Q(v,u) \cup P (u,x_3).$
 
We can now assume that  $V(Q) \cap V(Q')= \{v,w\}$ with $v \neq w$.
As before, suppose that $Q,$ $Q'$ are $1$-bunches with joints $u\in\{x_0,x_1\}$ and $u'\in\{x'_0,x'_1\}$.
Observe that $v,w$ are at distance one of $u,u'$. In this case, we see that one of the two paths in a 
$4$-pc~decomposition
of $H\cup H'$ is $ P(x_0,u) \cup Q(u,v) \cup [Q'-Q'(u',v)] \cup P'(u',x'_3).$
In the case that $Q'$ is a $2$-bunch with joints $x'_0,x'_2$, there exists a $4$-pc decomposition into two paths, 
with one of the paths as follows: $ P(x_0,u) \cup Q(u,w) \cup [Q'-Q'(w,x'_2)] \cup P'(x'_0,x'_2).$ \end{proof}


{ 
  \begin{lemma} \label{lemma:treepath} Let $T$ be a skeleton and
    $P$ be a path of length at least 4 such that $T\cup P$ is a
    tree, $E(T)\cap E(P)=\emptyset$ and $V(T)\cap V(P)=\{v\}$, where $v$
    is not a leaf of $P$.  If $T\cup P$ is
    not a skeleton and $d_o(T\cup P)\geq 3$, then it has a \hbox{4-pc decomposition.}
\end{lemma}
}
\begin{proof}{ Let $\lambda$ be the coloring of $T$.  By
    Observation~\ref{obs:1} and Proposition~\ref{prop:fromanypath},
    there exists a brrb-path $P'$ such that 
    $P'$ is the starting path of some building sequence of~$T$, and if
    $\lambda(v)$ is black (resp. red), then $v$ is an end vertex of
    $P'$ (resp. an inner vertex of $P'$).  Since $P'$ is a starting
    path, it suffices to prove that $P'\cup P$ has a $4$-pc~decomposition.  
    We consider the partition
    $P_1, P_2$ of~$P$, where $P_1$ and $P_2$ have $v$ as an end
    vertex.  In the case that $\lambda(v)$ is black, we have that $v$
    is an odd vertex in $T$, and since $d_o(T\cup P)\geq 3$, we have
    that $P_i$ has length at least~3, for each $i=1,2$.  If
    $\ell(P_1)=\ell(P_2)=3$, then $T\cup P$ is a skeleton, a case
    which we do not need to analyze.  Therefore, without loss of
    generality we can assume that $\ell(P_1)\geq 4$, and then, $P'\cup
    P$ can be decomposed into $P_1$ and $P' \cup P_2$.  Otherwise,
    $\lambda(v)$ is red, and then $v$ is an even vertex in $T$ and is at
    distance one (in $T$) of an odd vertex. And since $d_o(T\cup P)\geq 3$, we
    have that $P_i$ has length at least~2, for each $i=1,2$.  If
    $\ell(P_1)=\ell(P_2)=2$, then $T\cup P$ is a skeleton, a case which
    we do not need to analyze.  Hence, without loss of generality we can
    assume that $\ell(P_1)\geq 3$, and then, denoting by $v'$ the end
    vertex of $P'$ at distance one of $v$, we have that $P'\cup P$
    can be decomposed into $P_1\cup vv'$ and $(P'-vv') \cup P_2$.}
\end{proof}

\begin{lemma}\label{lemma:2cycles}
Let $G \in \mathcal{G}$ be a graph which is not a hanging square graph. If $Q_1$ and $Q_2$ are two
edge-disjoint squares in $G$ such that $G-E(Q_1)$ and $G-E(Q_2)$ are hanging-square
graphs, then $G$ has a \hbox{4-pc decomposition.}
\end{lemma}
\begin{proof}
  Let $H= G-E(Q_2)$ and {$V_1 = V(Q_1) \setminus V(T_H)$.} If $V_1 \cap V(Q_2)\neq\emptyset$, then
  the result follows by Lemma~\ref{lemma:length4}. Therefore, we assume that $V_1\cap
  V(Q_2)=\emptyset$.  Since $G$ is not a hanging-square graph, we can assume
  that there exists a brrb-path $P$ in $H$ such that $V(P)\cap V(Q_2)\neq\emptyset$ and
  $V(P)\cap V(Q_1)\neq\emptyset$.  
  By {Proposition~\ref{prop:fromanypath},} $P$ is the starting path of some
  building sequence $P, P_1, \ldots, P_t$ of $T_H$.
  Without loss of generality, we can assume that 
  $P=x_0x_1x_2x_3$ is occupied at $x_0,x_2$ by $Q_1$, in $H$.  Then, either
\begin{itemize}
 \item[(i)] $P$ is occupied at $x_1,x_3$ by $Q_2$, in $G-E(Q_1)$, or
 \item[(ii)] $Q_2$ is a 1-bunch with joint $x_1$, or
 \item[(iii)] $Q_2$ is a 1-bunch with joint $x_0$
   (resp. $x_2$) 
 and there exists $D$, a building graph in $\{P_1, \ldots, P_t\}$ 
 or a 1-bunch, with joint
 $x_2$ (resp. $x_0$) .

 \end{itemize}
In cases (i) and (ii) it is easy to prove that $P \cup Q_1 \cup Q_2$
has a decomposition into a path of length 5 and a path of length 6. In case~(iii), it can be shown
that $P \cup Q_1 \cup Q_2 \cup D$ decomposes into $3$~paths of length at
least~$4$. Further details of the proof are left to the reader.
\end{proof}

\section{Concluding remarks}\label{sec:concludremarks}

It would be interesting to broaden further the class of graphs for
which properties on path (or path-cycle) decompositions can be
well-characterized. In this direction, the study of the class of
triangle-free graphs with odd distance at least~$2$ is a challenging
problem. It is not so likely that a nice characterization (as in the
case of odd distance at least~$3$) can be found, but it would be
interesting to study path decomposition properties of this class of
graphs.

{ Finally, we note that the class of hanging-square graphs can be
recognized in polynomial time. To see this, let us say that a square
in a graph $G$ is \emph{good} if it has either $2$~vertices or
$3$~vertices of degree~$2$ in $G$. Given a triangle-free graph~$G$
with odd distance at least~$3$, we look for a square and check whether
it is good. If yes, we look for a maximal bunch, say $B$, that
contains it, and keep the information on its joints. Then, we delete
the edges of $B$ and repeat the process considering the resulting
graph, while it is connected. If the resulting graph is not connected
or if it contains cycles but none of them is a good square, then we
can conclude that the original graph is not a hanging-square graph. If
the resulting graph is a tree, say $T$, then we have to check whether
it is a skeleton. It is not difficult to see that the latter step can
be done in polynomial time, and that if $T$ is a skeleton, then we can
find in polynomial time a building sequence $\mathcal{S}=P_0, P_1,
\ldots, P_t$ of $T$.  If $T$ is not a skeleton, then the original
graph is not a hanging-square graph.  If $T$ is a skeleton, we have to
check whether all maximal bunches, whose edges were deleted
previously, have its joints intersecting properly (according to
Definition~\ref{def:h-sss}) the paths $P_i$ of $\mathcal{S}$.}

{We only sketched the ideas behind a polynomial-time recognition
  algorithm for the hanging-square graphs, as this algorithmic aspect
  is not the focus of this paper, but we wanted to discuss its
  consequence. We note that, in view of Theorem~\ref{theo:main}, 
the problem of deciding whether a graph
   in $\mathcal{G}$ admits a \hbox{$4$-pc} decomposition can be solved in
  polynomial time (the certificate that it belongs to the class $\coNP$
  being precisely the certificate that it is a hanging-square graph).}


\bigskip

\acknowledgements  We would like to thank F\'abio Botler for extensive discussions,
and the referees, for the very careful reading and for pointing out spots that
needed correction or clarification. 
\nocite{*}

\bibliography{biblio_gallai}

\end{document}